\documentclass[12pt]{article}

\usepackage{latexsym,amsfonts,amsmath,amssymb,epsfig,tabularx,amsthm,dsfont,mathrsfs}
\usepackage{graphicx}
\usepackage{hyperref}
\usepackage{enumerate}
\usepackage{color}
\usepackage{enumitem}

\allowdisplaybreaks

\theoremstyle{definition}

\newtheorem{thm}{Theorem}[section]
\newtheorem{rem}[thm]{Remark}

\newtheorem{lem}[thm]{Lemma}
\newtheorem{prop}[thm]{Proposition}
\newtheorem{cor}[thm]{Corollary}

\newtheorem{assumption}[thm]{Assumption}

\newcommand{\esup}{\operatornamewithlimits{ess\,sup}}

\renewcommand{\d}{{\rm d}}

\newcommand{\norm}[1]{\left\Vert #1 \right\Vert}

\newcommand{\N}{\mathbb{N}}
\newcommand{\E}{\mathbb{E}}
\newcommand{\R}{\mathbb{R}}

\newcommand{\cX}{\mathcal{X}}

\DeclareMathOperator{\Id}{\text{Id}}

\title{Optimal convergence rates of MCMC integration for functions with unbounded second moment}

\author{Julian Hofstadler\thanks{{Faculty of Computer Science and Mathematics, University of Passau, Innstraße 33, 94032 Passau, Germany.\\ Email: julian.hofstadler@uni-passau.de}}} 

\date{\today}
\begin{document}
	\maketitle

	\begin{abstract}
		We study the Markov chain Monte Carlo (MCMC) estimator for numerical integration for functions that do not need to be square integrable w.r.t. the invariant distribution. 
		For chains with a spectral gap we show that the absolute mean error for $L^p$ functions, with $p \in (1,2)$, decreases like $n^{1/p -1}$, which is known to be the optimal rate. 
		This improves currently known results where an additional parameter $\delta>0$ appears and the convergence is of order $n^{(1+\delta)/p-1}$. 
	\end{abstract}	
	
	
	\noindent
	\textbf{Classification:} 65C05; 60J22; 65C20
	
	\section{Introduction}
Let $(\cX, \mathcal{F}_\cX, \pi) $ be a probability space and $f \colon \cX \to \R $ measurable as well as $\pi$-integrable.
For a random variable $X\sim \pi$ we are interested in approximating the expectation 
\[
\E[f(X)] = \int_\cX f(x) \pi(\d x) = \pi(f).
\] 
A common approach is to use a Markov chain Monte Carlo (MCMC) method.
Requiring the density of $\pi$ only in non-normalised form, many MCMC algorithms provide powerful tools in scientific and statistical applications. 
The main idea behind these approaches is to construct a Markov chain $(X_n)_{n \in \N}$, having $\pi$ as the invariant distribution, and to estimate $\pi(f)$ via
\[
S_n f = \frac{1}{n}\sum_{j=1}^n f(X_j).
\] 

Under mild conditions we have $S_n f \to \pi(f)$ almost surely as $n \to \infty$, cf.~\cite{ASMUSSEN20111482} or \cite[Chapter 17]{meyn2012markov}.
This ensures the strong consistency of the MCMC estimator, yet, it is clearly of interest to have non-asymptotic error bounds.  
For instance, given some $p \in [1,\infty)$, one may consider the $p$-mean error
\begin{equation}\label{equ:p_mean_error}
	\E \left[ \left\vert S_n f - \pi(f) \right\vert^p \right],
\end{equation}
however, also other criteria are feasible, see e.g. \cite{kunsch2019optimal}.

Setting $p=2$ in \eqref{equ:p_mean_error} we speak of the \textit{mean squared error}, and in different settings explicit bounds are known, e.g. under a Wasserstein contraction assumption \cite{Joulin2010Curvature}, spectral gap conditions \cite{rudolf2012explicit}, or if one has geometric/polynomial ergodicity \cite{latuszynski2015nonasymptotic}.
For the mean squared error to make sense we require a finite second moment of $f$, i.e. $\norm{f}_{L^2(\pi)}^2 = \pi(f^2) < \infty$. On the other hand, the \textit{absolute mean error}, given by $\E \vert S_n f - \pi (f) \vert$, is well defined and finite as long as $f$ is $\pi$-integrable. 
Bounds for the absolute mean error for functions with $\norm{f}_{L^p(\pi)}^p = \pi (\vert f \vert^p)< \infty $, where $p<2$, are still rare. 
In \cite{RUDOLF2015Error} it is shown that under a spectral gap condition for any $p \in (1,2)$ holds
\begin{equation}\label{equ:upper_bound_error}
	\sup_{\norm{f}_{L^p(\pi)}\leq 1} \E[\vert S_n f - \pi(f)\vert ] \leq \frac{C}{n^{1-\frac{1+\delta}{p}}},
\end{equation}
with constants $\delta >0$ and $C\in (0, \infty)$.

Proposition 1 of Section 2.2.9 (with $k=0$) in \cite{novak1988deterministic}, see also \cite[Section 5]{He94},  shows that in general we have the following lower bound
\begin{equation}\label{equ:lower_bound_error}
	\sup_{\norm{f}_{L^p(\pi)}\leq 1} \E[\vert S_n f - \pi(f)\vert ] \geq \frac{c}{n^{1-\frac{1}{p}}},
\end{equation} 
where $c \in (0, \infty)$ is a constant independent of $n$.

Even though $\delta>0 $ in \eqref{equ:upper_bound_error} may be chosen arbitrarily small, it is natural to ask whether it can be removed completely, such that one would have the same rate as in \eqref{equ:lower_bound_error}. 
Under the (strong) assumption of uniform ergodicity and reversibility we know that this is the case, cf. \cite[Theorem 1]{RUDOLF2015Error}. 
To the best of the author's knowledge this is the only situation where optimal rates are known. 
The goal of this note is to extend this result to the spectral gap setting, see Theorem \ref{thm:error_bound_p_mean}, where we show
\[
\sup_{\norm{f}_{L^p(\pi)}\leq 1} \E[\vert S_n f - \pi(f)\vert ] \leq \frac{\widetilde{C}_p}{n^{1-\frac{1}{p}}},
\]
for $p \in (1,2]$, with an explicit expression for the constant $\widetilde{C}_p$.

Let us sketch the proof: The main idea is to employ the Riesz-Thorin interpolation theorem, a technique which goes back at least to \cite{He94} in Monte Carlo theory and was also used to derive \eqref{equ:upper_bound_error} in \cite{RUDOLF2015Error}.  
Thereto we first derive a result for the case where $(X_n)_{n \in \N}$ is a stationary chain, see Proposition \ref{prop:L_p_stationary}.
	Then we apply a change of measure argument to deduce Theorem \ref{thm:error_bound_p_mean}.
	It is worth mentioning that based on Proposition \ref{prop:L_p_stationary} it is also possible to generalise \cite[Theorem 3.41]{rudolf2012explicit}, see Corollary \ref{cor:p_mean_error}.

The rest of this note is organised as follows: In Section \ref{sec:errorbounds} we state and discuss our assumptions as well as main results. The proofs, together with some intermediate results, can be found in Section \ref{sec:proofs}.

	\section{Error bounds for MCMC integration}\label{sec:errorbounds}
This section contains our main result, see Theorem \ref{thm:error_bound_p_mean}, together with all required notation and assumptions. 
Let us start with specifying the general setting. 

We assume that the state space $(\cX, \mathcal{F}_\cX)$ is Polish with $\mathcal{F}_\cX$ being countably generated. 
Let $K$ be a Markov kernel and $\nu$ a probability measure, called the initial distribution, both defined on $(\cX, \mathcal{F}_\cX)$. 
Then, the Markov chain corresponding to $K$ and $\nu$, say $(X_n)_{n \in \N}$,  is defined on a probability space $(\Omega, \mathcal{F}, \mathbb{P}_\nu)$.
In particular, such a probability space exists.
We assume that $\pi$ is the unique invariant distribution of $K$ and that $(X_n)_{n \in \N}$ is $\psi$-irreducible. 
For definitions and further details we refer to \cite{douc2018markov,meyn2012markov}.

{Given $p \in [1, \infty)$, we define $L^p (\pi)$ as the set of all functions $f \colon \cX \to \R$ such that $\norm{f}_{L^p(\pi)}^p = \pi(\vert f \vert^p)< \infty$. Similarly, $L^\infty(\pi)$ denotes the set of functions with finite $\norm{f}_{L^\infty(\pi)} = \esup_{x \in \cX} \vert f(x) \vert$.
	We follow the usual convention that two functions in $L^p(\pi)$, with $p \in [1, \infty]$, are considered as equal if they are equal $\pi$-almost everywhere.
	Then, $(L^p(\pi),\norm{\cdot}_{L^p(\pi)})$ is a normed space.
	Moreover, $L^2(\pi)$, equipped with $\langle f, g\rangle_{L^2(\pi)} = \int_\cX f(x) g(x) \pi(\d x)$, is a Hilbert space with induced norm $\norm{\cdot }_{L^2(\pi)}$, and so is the following closed subspace 
	\[
	L^2_0 (\pi) := \left\{  f \in L^2(\pi)  \colon \pi(f) =0 \right\}.
	\]
}
Let $p \in [1, \infty]$. A Markov kernel $K$ induces an operator $K\colon L^p(\pi)\to L^p(\pi)$ via $f \mapsto Kf(\cdot) = \int_\cX f(x') K(\cdot , \d x')$. 
Indeed, the operator $K$ is well defined, linear, and one has $\norm{K}_{L^p(\pi) \to L^p(\pi)} =1$, we refer to \cite[Section 3.1]{rudolf2012explicit} for further details.

By $\Id $ we denote the identity on $L^2(\pi)$. 
The following condition about the operator $\Id-K$, restricted to $L^2_0(\pi)$, is our main assumption. 
\begin{assumption}\label{assumption_1}
	Assume that $\Id - K$, considered as operator from  $L^2_0(\pi)$ to $L^2_0 (\pi)$, has a linear and bounded inverse with 
	\[
	\norm{(\Id - K)^{-1}}_{L^2_0(\pi) \to L^2_0(\pi)} \leq s < \infty.
	\]
\end{assumption}

\begin{rem}
	The invertibility of $\Id -K$, restricted to a suitable subspace of $L^2(\pi)$, was also studied in \cite{mathe1999numerical} for uniformly ergodic chains and \cite{mathe2004numerical} for $V$-uniformly ergodic chains.
	In particular, the existence of $(\Id - K)^{-1}$ on an appropriate subspace was used there to characterise the convergence behaviour of the mean squared error.
	Moreover, non-reversible chains on finite state spaces were studied recently in \cite{chatterjee2023spectral}. There, bounds for the mean squared error are shown, where the second smallest singular value of $\Id -K$ plays an important role.
\end{rem}

We note that Assumption \ref{assumption_1} is closely related to a spectral gap, there are, however, different definitions for a spectral gap:
\begin{itemize}
	\item Some authors, see e.g. \cite{andrieu2023explicit,douc2018markov}, say $K$ admits a(n) (absolute $L^2$) spectral gap if $\sup_{\lambda \in \mathcal{S}_0} \vert \lambda \vert <1$, where $\mathcal{S}_0$ is the spectrum of $K \colon L^2_0 (\pi) \to L^2_0 (\pi)$. 
	This is equivalent to the existence of some $m \in \N$ such that $\norm{K^m}_{L^2_0(\pi) \to L^2_0(\pi)} <1$, cf. \cite[Proposition 22.2.4]{douc2018markov}.
	\item On the other hand, a different definition, for instance used in \cite{hairer2014spectral,natarovskii2021quantitative,RUDOLF2015Error,rudolf2012explicit},  is to say that $K$ admits a(n) (absolute $L^2$) spectral gap if $\norm{K}_{L^2_0(\pi)\to L^2_0(\pi)} <1$. 
\end{itemize}

If $K$ is reversible, which implies that the corresponding Markov operator is self-adjoint on $L^2(\pi)$, then both definitions are equivalent.

Let us emphasize that either of the above definitions of a spectral gap implies that Assumption \ref{assumption_1} is true, also in the non-reversible case. 
Spectral gap results were established for a number of MCMC methods, see for instance \cite{andrieu2023explicit,hairer2014spectral,natarovskii2021quantitative},
see also \cite[Section 3.4]{rudolf2012explicit} and \cite[Theorem 2.1]{roberts1997geometric}. Moreover, we note that under Assumption \ref{assumption_1} we cover the setting of Theorems~1 and 2 of \cite{RUDOLF2015Error}. 

If for the initial distribution we have $\nu \ll \pi$ with Radon-Nikod{\'y}m derivative $\frac{\d \nu}{\d \pi}\in L^q(\pi)$ for some $q \in[1, \infty]$, then we set $M_q = \norm{\frac{\d \nu}{\d \pi}}_{L^q(\pi)}$. 
	In particular, for $q = \infty$ we have $\sup_{A \in \mathcal{F}_\cX}\frac{\nu(A)}{\pi(A)} \leq M_\infty$, in which case $\nu$ is called $M_\infty$-warm.
	
	The following theorem is our main result, which shows that under Assumption \ref{assumption_1} we have the optimal rate of convergence for the absolute mean error.

\begin{thm}\label{thm:error_bound_p_mean}
		Let Assumption \ref{assumption_1} be true, $p \in (1,2]$, and assume that $\nu$ is absolutely continuous w.r.t. $\pi$ with Radon-Nikod{\'y}m derivative $\frac{\d \nu}{\d \pi}\in L^q(\pi)$, where $p^{-1} + q^{-1} =1$. 
		Then, for any $f \in L^p(\pi)$ and any $n \in \N$ we have that
		\[
		\E_\nu \left[\left\vert S_n f - \pi(f)\right\vert  \right] \leq \frac{C_p M_q \norm{f}_{L^p(\pi)}}{n^{1-1/p} }    ,
		\] 
		where $C_p = 2^{2/p-1} \cdot (8s^2)^{1-1/p}$ and $M_q = \norm{\frac{\d \nu}{\d \pi}}_{L^q(\pi)}$.
\end{thm}

\begin{rem}
	We note that Theorem \ref{thm:error_bound_p_mean} is still true, with the same bound, if we replace $S_nf$ by $S_{n, n_0}f = \frac{1}{n}\sum_{j=1}^{n} f(X_{j+n_0})$, which corresponds to using a \textit{burn in} of length $n_0 \in \N$. 
\end{rem}

\begin{rem}
		Theorem \ref{thm:error_bound_p_mean} shows that for $p\in (1,2]$ and $q \in [1, \infty)$ with $p^{-1} + q^{-1}=1$, and $f \in L^p(\pi)$ we have
		\[
		\E_\nu \left[\left\vert S_n f - \pi(f)\right\vert  \right] \leq \frac{C_p M_q \norm{f}_{L^p(\pi)}}{n^{1-1/p} }   ,
		\]
		with $C_p, M_q$ as specified in the theorem.
		Given $c \in \R$ we set $f_c = f+c$, and note that $S_n f_c - \pi(f_c) = S_n f - \pi(f)$.
		Thus, for $\E_\nu \left[\left\vert S_n f_c - \pi(f_c)\right\vert  \right]$ we have the same bound as above, even though $\norm{\cdot}_{L^p(\pi)}$ is not invariant w.r.t. linear shifts. 
\end{rem} 

\begin{rem}
	In Theorem~\ref{thm:error_bound_p_mean} the quantity $\norm{\frac{\d \nu}{\d \pi}}_{L^q(\pi)}$ appears.
	In some sense this penalises our choice of the initial distribution $\nu$ which is allowed to differ from the target $\pi$. 
	In the setting where a fixed computational budget is available it may be worth spending some effort to find a ``good" initial distribution $\nu$. 
	However, discussing optimal choices of $\nu$ w.r.t. different theoretical and/or practical aspects is beyond the scope of this note.
\end{rem}

	\section{Proofs}\label{sec:proofs}
In this section we prove our main results. 
Recall that the chain $(X_n)_{n \in \N}$ is defined on the probability space $(\Omega, \mathcal{F}, \mathbb{P}_\nu)$.
For $p \in [1, \infty]$ let $L^p(\mathbb{P}_\nu)$ be the set of random variables $Y$ on $(\Omega, \mathcal{F}, \mathbb{P}_\nu)$ with $\norm{Y}_{L^p(\mathbb{P}_\nu)}^p =\E_\nu [\vert Y \vert^p]< \infty$. Similar as for the $L^p(\pi)$ spaces we consider $Y_1, Y_2 \in L^p(\mathbb{P}_\nu)$ as equal if $Y_1=Y_2$ holds $\mathbb{P}_\nu$-almost surely.
Then, $(L^p(\mathbb{P}_\nu), \norm{ \cdot}_{L^p(\mathbb{P}_\nu)})$ is a normed space.
Moreover, the subspace $L_0^p(\mathbb{P}_\nu)\subseteq L^p(\mathbb{P}_\nu)$ contains all $Y \in L^p(\mathbb{P})$ such that $\E_\nu [Y] = 0$.

The first result of this section provides a bound for the mean squared error of $S_n h$ for the case where $(X_n)_{n \in \N}$ is stationary, i.e. where $X_1 \sim \pi$, and $h $ is a centred function.  

\begin{lem}\label{lem:MSE_stationary}
	Let Assumption \ref{assumption_1} be true and assume that $(X_n)_{n \in \N}$ has initial distribution $\pi$, i.e. $X_1 \sim \pi$. 
	Then, for any $h \in L^2_0(\pi)$ and any $n \in \N$ we have
	\[
	\E_\pi \left[ \vert S_n h \vert^2  \right] \leq \frac{8 s^2}{n} \norm{h}_{L^2(\pi)}^2.
	\] 
\end{lem}

\begin{proof}
	Expanding $\E_\pi \left[ \vert S_n h \vert^2  \right] $ and using \cite[Lemma 3.25]{rudolf2012explicit} we get 
	\[
	\E_\pi \left[ \vert S_n h \vert^2  \right] 
	= \frac{1}{n^2} \sum_{j=1}^n \sum_{k=1}^n \E_\pi [h(X_j) h(X_k) ] = \frac{1}{n^2} \sum_{j=1}^{n} \sum_{k=1}^n \left\langle h, K^{\vert j-k\vert } h \right\rangle_{L^2(\pi)}.
	\]
	One can prove that on $L^2_0(\pi)$ the following identity holds 
	\begin{align}\label{equ:operator_identiy}
		(\Id-K)^2\sum_{j=1}^n \sum_{k=1}^nK^{\vert j-k\vert} 
		= n(1-K^2) - 2K(1-K^n).
	\end{align}
	Here for $n\geq 2$ one may use that $\sum_{j=1}^n \sum_{k=1}^nK^{\vert j-k\vert} = n \Id + 2\sum_{\ell=1}^{n-1} (n-\ell)K^\ell$ and the formula $\sum_{\ell=1}^{n-1} (n-\ell)K^\ell = (\Id - K)^{-2} (K^{n+1}- nK^2 +(n-1)K)$, which can be shown by induction. For $n=1$ the identity follows by plugging in. 
	
	 Note that $\left\langle h, K^{\vert j-k\vert } h \right\rangle_{L^2(\pi)} = \left\langle h, K^{\vert j-k\vert } h \right\rangle_{L^2_0(\pi)}$. 
	 Using this and \eqref{equ:operator_identiy} we get
	\[
	\E_\pi\left[  \vert S_n h\vert^2 \right] = \frac{1}{n^2} \sum_{j=1}^{n} \sum_{k=1}^n \left\langle h, K^{\vert j-k\vert } h \right\rangle_{L^2_0(\pi)}
	= \frac{1}{n^2}\left\langle h, (\Id - K)^{-2} g \right\rangle_{L^2_0(\pi)},
	\]
	where we set $g = \left(n(\Id-K^2) - 2 K(\Id -K^n) \right) h \in L^2_0(\pi)$.
	By triangle inequality, properties of operator norms, and since $\norm{K^\ell}_{L^2_0(\pi) \to L^2_0(\pi)} \leq 1$ for any $\ell \in \N$ it follows that
	\begin{align*}
		\norm{g}_{L^2_0(\pi)} &\leq n\norm{(\Id - K^2) h}_{L^2_0(\pi)} + 2 \norm{K(\Id - K^n) h}_{L^2_0(\pi)} \\
		&\leq 
		 \left( n\norm{(\Id - K^2)}_{L^2_0(\pi) \to L^2_0(\pi)}  + 2 \norm{K(\Id - K^n)}_{L^2_0(\pi) \to L^2_0(\pi)}\right)  \norm{ h}_{L^2_0(\pi)} \\
		&\leq 
		\left( 2n  + 4 \right) \norm{ h}_{L^2_0(\pi)} \\ 
		&\leq 
		4(n+1) \norm{h}_{L^2_0(\pi)}. 
	\end{align*}
	Thus, by Cauchy-Schwarz inequality we obtain 
	\begin{align*}
		\E_\pi [\vert S_n h \vert^2 ] = \frac{1}{n^2} \left\langle h, (\Id - K)^{-2} g \right\rangle_{L^2_0(\pi)} \leq \frac{s^2}{n^2} \norm{h}_{L^2_0(\pi)} \norm{g}_{L^2_0(\pi)} \\
		\leq 
		\frac{4(n+1) s^2}{n^2} \norm{h}_{L^2_0(\pi)}^2  \leq \frac{8 s^2}{n} \norm{h}_{L^2_0(\pi)}^2.
	\end{align*}
	Since $\norm{h}_{L^2_0(\pi)} = \norm{h}_{L^2(\pi)}$ the claimed bound follows.
\end{proof}

\begin{prop}\label{prop:L_p_stationary}
	Let Assumption \ref{assumption_1} be true, let $p \in [1, 2]$ and assume that $(X_n)_{n \in \N}$ has initial distribution $\pi$, i.e. $X_1 \sim \pi$. 
	Then, for any $f \in L^p(\pi)$ and any $n \in \N$ we have
	\[
	\E_\pi[\vert S_n f- \pi(f) \vert^p  ] \leq 2^{2-p} \left(\frac{8s^2}{n}\right)^{p-1} \norm{f}_{L^p(\pi)}^p.
	\]
\end{prop}

\begin{proof}
	We set $\bar{f} = f - \pi(f) \in L^2_0(\pi)$ and $T_n f= S_nf - \pi(f)$. 
	Hence by Lemma \ref{lem:MSE_stationary} and the fact that $\norm{\bar{f}}_{L^2(\pi)}^2 = \pi(f^2) - \pi(f)^2 \leq \pi(f^2) = \norm{f}_{L^2(\pi)}^2  $ we have
	\[
	\E_\pi [\vert T_n f\vert^2 ] = \E_\pi[\vert S_n \bar{f} \vert^2 ] \leq  \frac{8s^2}{n} \norm{\bar{f}}_{L^2(\pi)}^2 \leq \frac{8s^2}{n} \norm{f}_{L^2(\pi)}^2,
	\]
	which implies $\norm{T_n}_{L^2(\pi) \to L^2(\mathbb{P}_\pi)} \leq \sqrt{\frac{8s^2}{n}}$.
	Moreover, using triangle inequality we see that $\norm{T_n}_{L^1(\pi) \to L^1(\mathbb{P}_\pi)} \leq 2$.
	Thus, by the Riesz-Thorin interpolation theorem, see \cite[Theorem 1.3.4]{grafakos2014classical} in the setting $p_0 = q_0 = 1$; $p_1 = q_1 = 2$ and $\theta = 2- 2/p$, we obtain that 
	\[
	\norm{T_n}_{L^p(\pi) \to L^p(\mathbb{P}_\pi)} \leq 2^{2/p-1}\left(\frac{8s^2}{n}\right)^{1-1/p},
	\]  
	which implies the result.
\end{proof}

Using a change of measure argument and Proposition \ref{prop:L_p_stationary} we are able to prove our main result. 
However, before we turn to the proof of Theorem~\ref{thm:error_bound_p_mean} let us state another consequence of Proposition \ref{prop:L_p_stationary}.
We note that the upcoming result generalises \cite[Theorem 3.41]{rudolf2012explicit} by providing a bound for the mean squared error of $S_n f$ for $L^2(\pi)$ functions.

\begin{cor}\label{cor:p_mean_error}
	Let Assumption \ref{assumption_1} be true, let $p \in [1, 2]$ and assume that $\nu \ll \pi$ with Radon-Nikod{\'y}m derivative $\frac{\d \nu}{\d \pi}\in L^\infty(\pi)$. 
	Then, for any $f \in L^p(\pi)$ and any $n \in \N$ we have
	\[
	\E_\nu \left[ \vert S_n f - \pi(f)  \vert^p \right] \leq \frac{C_p M_\infty \norm{f}_{L^p(\pi)}^p}{n^{p-1}} ,
	\]
	where $C_p = 2^{2-p}\cdot(8s^2)^{p-1}$ and $M_\infty = \norm{\frac{\d \nu}{\d \pi}}_{L^\infty(\pi)}$.
\end{cor}

\begin{proof}
	Since $\mathbb{P}_\nu$ and $\mathbb{P}_\pi$ only differ by the choice of the initial distribution of $(X_n)_{n \in \N}$ we have
	\[
	\E_\nu [\left\vert S_n f - \pi(f) \right\vert^p ] =  \E_\pi \left[\frac{\d \nu}{\d \pi} (X_1)\left\vert S_n f - \pi(f) \right\vert^p \right] 
	\leq 
	\norm{\frac{\d \nu}{\d \pi}}_{L^\infty(\pi)} \E_\pi [\left\vert S_n f - \pi(f) \right\vert^p ].
	\]
	Hence, the result follows from Proposition \ref{prop:L_p_stationary}. 
\end{proof}
Finally, we turn to the proof of Theorem \ref{thm:error_bound_p_mean}

\begin{proof}[Proof of Theorem \ref{thm:error_bound_p_mean}] 
	By the same change of measure argument as before and Hoelder's inequality
	\[
	\E_\nu [\left\vert S_n f - \pi(f) \right\vert ] =   \E_\pi \left[\frac{\d \nu}{\d \pi} (X_1)\left\vert S_n f - \pi(f)\right\vert \right] 
	\leq 
	\norm{\frac{\d \nu}{\d \pi}}_{L^q(\pi)} \E_\pi [\left\vert S_n f - \pi(f) \right\vert^p ]^{1/p}.
	\]
	Now the desired result is a consequence of Proposition \ref{prop:L_p_stationary}. 
\end{proof}
	
	\section*{Acknowledgements}
	The author would like to thank Mareike Hasenpflug and Daniel Rudolf for valuable discussions which helped to significantly improve the presentation of the paper. 
	Moreover, the author thanks two anonymous referees for reading the paper carefully, and suggesting a simplification for Section \ref{sec:proofs}.
	JH is supported by the DFG within project 432680300 -- SFB 1456 subproject B02.
	
\providecommand{\bysame}{\leavevmode\hbox to3em{\hrulefill}\thinspace}
\providecommand{\MR}{\relax\ifhmode\unskip\space\fi MR }
\providecommand{\MRhref}[2]{%
	\href{http://www.ams.org/mathscinet-getitem?mr=#1}{#2}
}
\providecommand{\href}[2]{#2}


\begin{thebibliography}{ALPW23}
	
	\bibitem[AG11]{ASMUSSEN20111482}
	Søren Asmussen and Peter~W. Glynn, \emph{A new proof of convergence of {MCMC}
		via the ergodic theorem}, Statistics \& Probability Letters \textbf{81}
	(2011), no.~10, 1482--1485.
	
	\bibitem[ALPW23]{andrieu2023explicit}
	Christophe Andrieu, Anthony Lee, Sam Power, and Andi~Q. Wang, \emph{{Explicit
			convergence bounds for Metropolis Markov chains: isoperimetry, spectral gaps
			and profiles}}, arXiv preprint
	\href{https://arxiv.org/pdf/2211.08959.pdf}{arXiv:2211.08959}, 2023.
	
	\bibitem[Cha23]{chatterjee2023spectral} 
	 Sourav Chatterjee,
	\emph{Spectral gap of nonreversible Markov chains}.  arXiv preprint \href{https://arxiv.org/pdf/2310.10876}{arXiv:2310.10876}, 2023.
	
	\bibitem[DMPS18]{douc2018markov}
	Randal Douc, {\'E}ric Moulines, Pierre Priouret, and Philippe Soulier,
	\emph{Markov chains}, Springer, 2018.
	
	\bibitem[Gra14]{grafakos2014classical}
	Loukas Grafakos, \emph{Classical {F}ourier analysis}, third ed., Graduate Texts
	in Mathematics, vol. 249, Springer, 2014.
	
	\bibitem[Hei94]{He94}
	Stefan Heinrich, \emph{Random approximation in numerical analysis}, Proceedings
	of the Conference ``Functional Analysis'' Essen (1994), 123--171.
	
	\bibitem[HSV14]{hairer2014spectral}
	Martin Hairer, Andrew~M. Stuart, and Sebastian~J. Vollmer, \emph{Spectral gaps
		for a {M}etropolis–{H}astings algorithm in infinite dimensions}, The Annals
	of Applied Probability \textbf{24} (2014), no.~6, 2455--2490.
	
	\bibitem[JO10]{Joulin2010Curvature}
	Ald{\'e}ric Joulin and Yann Ollivier, \emph{{Curvature, concentration and error
			estimates for Markov chain Monte Carlo}}, The Annals of Probability
	\textbf{38} (2010), no.~6, 2418 -- 2442.
	
	\bibitem[KR19]{kunsch2019optimal}
	Robert~J. Kunsch and Daniel Rudolf, \emph{Optimal confidence for {M}onte
		{C}arlo integration of smooth functions}, Advances in Computational
	Mathematics \textbf{45} (2019), no.~5, 3095--3122.
	
	\bibitem[{\L}MN13]{latuszynski2015nonasymptotic}
	Krzysztof {\L}atuszy{\'n}ski, B{\l}a{\.z}ej Miasojedow, and Wojciech Niemiro,
	\emph{Nonasymptotic bounds on the estimation error of {MCMC} algorithms},
	Bernoulli \textbf{19} (2013), no.~5A, 2033--2066.
	
	\bibitem[Mat99]{mathe1999numerical}
	Peter Mathé, \emph{Numerical integration using {Markov Chains}}, Monte Carlo
	Methods and Applications \textbf{5} (1999), no.~4, 325--343.
	
	\bibitem[Mat04]{mathe2004numerical}
	\bysame, \emph{Numerical integration using v-uniformly ergodic {Markov
			Chains}}, Journal of Applied Probability \textbf{41} (2004), no.~4,
	1104--1112.
	
	\bibitem[MT12]{meyn2012markov}
	Sean~P. Meyn and Richard~L. Tweedie, \emph{Markov chains and stochastic
		stability}, second ed., Cambridge University Press, 2012.
	
	\bibitem[Nov88]{novak1988deterministic}
	Erich Novak, \emph{Deterministic and stochastic error bounds in numerical
		analysis}, Lecture notes in Mathematics, vol. 1349, Springer, 1988.
	
	\bibitem[NRS21]{natarovskii2021quantitative}
	Viacheslav Natarovskii, Daniel Rudolf, and Bj{\"o}rn Sprungk,
	\emph{{Quantitative spectral gap estimate and Wasserstein contraction of
			simple slice sampling}}, The Annals of Applied Probability \textbf{31}
	(2021), no.~2, 806 -- 825.
	
	\bibitem[RR97]{roberts1997geometric}
	Gareth~O. Roberts and Jeffrey~S. Rosenthal, \emph{Geometric ergodicity and
		hybrid {M}arkov chains}, Electronic Communications in Probability \textbf{2}
	(1997), no.~none, 13 -- 25.
	
	\bibitem[RS15]{RUDOLF2015Error}
	Daniel Rudolf and Nikolaus Schweizer, \emph{{Error bounds of MCMC for functions
			with unbounded stationary variance}}, Statistics \& Probability Letters
	\textbf{99} (2015), 6--12.
	
	\bibitem[Rud12]{rudolf2012explicit}
	Daniel Rudolf, \emph{{Explicit error bounds for Markov chain Monte Carlo}},
	Dissertationes Math. \textbf{485} (2012).
	
\end{thebibliography}
\end{document}